\documentclass[12pt]{amsart}
\usepackage[utf8]{inputenc}
\usepackage{amssymb,latexsym,amsmath,extarrows}
\usepackage{xcolor}

\usepackage[margin=1.2in, centering]{geometry}

\newtheorem{theorem}{Theorem}[section]
\newtheorem{lemma}[theorem]{Lemma}
\newtheorem{proposition}[theorem]{Proposition}

\newtheorem{definition}[theorem]{Definition}

\newtheorem*{conjecture*}{Conjecture}

\newtheorem*{acknowledgement}{Acknowledgements}

\newcommand{\al}{\alpha}

\newcommand{\la}{\lambda}

\newcommand{\si}{\sigma}

\newcommand{\ZR}{\mathbb{R}}

\newcommand{\ZT}{\mathbb{T}}

\begin{document}

\title[Falconer's distance set problem in even dimensions]{An improved result for Falconer's distance set problem in even dimensions}

\author[X. Du]{Xiumin Du}
\author[A. Iosevich]{Alex Iosevich}
\author[Y. Ou]{Yumeng Ou} 
\author[H. Wang]{Hong Wang}
\author[R. Zhang]{Ruixiang Zhang}
\address[X. Du]{Department of Mathematics, Northwestern University, Evanston, IL 60208}
\address[A. Iosevich]{Department of Mathematics, University of Rochester, Rochester, NY 14627}
\address[Y. Ou]{Department of Mathematics, University of Pennsylvania, Philadelphia, PA  19104}
\address[H. Wang]{School of Mathematics, Institute for Advanced Study, Princeton, NJ 08540}
\address[R. Zhang]{School of Mathematics, Institute for Advanced Study, Princeton, NJ 08540}

\maketitle
\begin{abstract}
We show that if compact set $E\subset \mathbb{R}^d$ has Hausdorff dimension larger than $\frac{d}{2}+\frac{1}{4}$, where $d\geq 4$ is an even integer, then the distance set of $E$ has positive Lebesgue measure. This improves the previously best known result towards Falconer's distance set conjecture in even dimensions.
\end{abstract}

\section{Introduction}

Let $E\subset\mathbb{R}^d$ be a compact set, its \emph{distance set} $\Delta(E)$ is defined by
$$
\Delta(E):=\{|x-y|:x,y\in E\}\,.
$$
A classical question in geometric measure theory, introduced by Falconer in the early 80s (\cite{F}) is, how large does the Hausdorff dimension of a compact subset of ${\ZR}^d$, $d\ge 2$ need to be to ensure that the Lebesgue measure of the set of pairwise Euclidean distances is positive.

\begin{conjecture*}\label{conj} \textup{[Falconer]}
Let $d\geq 2$ and $E\subset\mathbb{R}^d$ be a compact set. Then
$$
{\rm dim}(E)> \frac d 2 \Rightarrow |\Delta(E)|>0.
$$
Here $|\cdot|$ denotes the Lebesgue measure and ${\rm dim}(\cdot)$ is the Hausdorff dimension.
\end{conjecture*}
This conjecture, still open in all dimensions, has a famous predecessor in discrete geometry known as the Erd\H{o}s distinct distance conjecture. It says that $N$ points in ${\ZR}^d$, $d \ge 2$, determine at least $C_{\epsilon}N^{\frac{2}{d}-\epsilon}$, $\epsilon>0$, distinct Euclidean distances. The two dimensional case was solved by Guth and Katz \cite{GK15} after more than half of century of partial results. The higher dimensional case is still open, with the best known exponents obtained by Solymosi and Vu \cite{SV08}. There are some intriguing connections between the Erd\H os and Falconer distance problem, the issue that we shall touch upon at the end of this paper.

The main purpose of this paper is to improve the best known dimensional threshold towards the Falconer conjecture in even dimensions.

\begin{theorem}\label{main}
Let $d\geq 4$ be an even integer and $E\subset\mathbb{R}^d$ be a compact set. Then
$$
{\rm dim}(E)> \frac d 2+\frac{1}{4} \Rightarrow |\Delta(E)|>0.
$$
\end{theorem}

Falconer's conjecture has attracted a great amount of attention in the past decades, and different methods have been invented to lower the dimensional threshold that is sufficient for the distance set to have positive Lebesgue measure. To name a few important landmarks in the study of the problem: in 1985, Falconer \cite{F} showed that $|\Delta(E)|>0$ if ${\rm dim}(E)>\frac{d}{2}+\frac{1}{2}$. Bourgain \cite{B} was the first to lower the threshold $\frac{d}{2}+\frac{1}{2}$ in dimensions $d=2, d=3$ and to use the theory of Fourier restriction in the Falconer problem. The thresholds were further improved by Wolff \cite{W99} to $\frac{4}{3}$ in the case $d=2$, and by Erdo\u{g}an \cite{Erdg05} to $\frac{d}{2}+\frac{1}{3}$ when $d\geq 3$. These records were only very recently rewritten in 2018:
\[
\begin{cases}
\frac{5}{4}, &d=2, \quad \text{(Guth--Iosevich--Ou--Wang \cite{GIOW})}\\
\frac{9}{5}, &d=3, \quad \text{(Du--Guth--Ou--Wang--Wilson--Zhang \cite{DGOWWZ})}\\
\frac{d^2}{2d-1}, &d\geq 4, \quad \text{(Du--Zhang \cite{DZ})}.
\end{cases}
\]
Our main result in this paper further improves the thresholds in even dimensions $d\geq 4$.

The proof of Theorem \ref{main} is inspired by many key ingredients in \cite{GIOW}, and the numerology $\frac{d}{2}+\frac{1}{4}$ matches the two dimensional case. Similarly to \cite{GIOW}, we in fact prove a slightly stronger version of the main theorem regarding the pinned distance set.

\begin{theorem}\label{thm: pinned}
Let $d\geq 4$ be an even integer and $E\subset\mathbb{R}^d$ be a compact set. Suppose that $\text{dim}(E)>\frac{d}{2}+\frac{1}{4}$, then there is a point $x\in E$ such that its pinned distance set $\Delta_x(E)$ has positive Lebesgue measure, where
$$
\Delta_x(E):=\{|x-y|:\, y\in E\}.
$$
\end{theorem}

Let $E\subset \mathbb{R}^d$ be a compact set with positive $\alpha$-dimensional Hausdorff measure. It is a standard result that there exists a probability measure $\mu$ supported on $E$ such that
\[
\mu(B(x,r))\lesssim r^\alpha,\qquad \forall x\in \mathbb{R}^d,\, \forall r>0.
\]
Such measure is usually referred to as a Frostman measure. In the study of the Falconer problem, a classical analytic approach due to Mattila \cite{M} is to reduce the desired result ($|\Delta(E)|>0$) to showing certain estimates of the decay rate of the Fourier transform of $\mu$. This is also precisely the route taken in many prior works including \cite{W99, Erdg05, DGOWWZ, DZ}.

However, this approach also has its limit. For instance, it is known that when $d=2$, the best possible Falconer threshold it could imply is $\frac{4}{3}$, which matches the result of Wolff \cite{W99}. And when $d=3$, $\frac{5}{3}$ would be the best possible (see \cite{Erdg05}). In higher dimensions, there is no currently known example showing such constraint. Whereas, in \cite{D} it shows that further constraints arise if the method employed does not distinguish the spherical and parabolic decay rates.

In \cite{GIOW}, the authors studied the two dimensional Falconer problem, and developed a new method that modifies the original Mattila approach. Their argument consists primarily of two steps. First, one prunes the natural Frostman measure $\mu$ on $E$ by removing \emph{bad} wave packets at different scales (see Section \ref{sec: good} below for the exact process), and shows that the error introduced in the pruning process can be controlled. Second, one applies a refined decoupling inequality to estimate some $L^2$ quantity involving the pruned good measure. One of the main reasons why the argument doesn't readily extend to higher dimensions is because of the first step. More precisely, in \cite{GIOW}, to make sure that the pruned measure is close enough to the original Frostman measure, one applies a radial projection theorem of Orponen \cite{O17b} (see Theorem \ref{thm: Orponen} below) that assumes the measure has dimension $\alpha>d-1$. However, when $d\geq 3$, this condition fails to hold if $\alpha$ is close enough to $\frac{d}{2}$. We overcome this difficulty by introducing another ingredient into the process: orthogonal projections of the original measure, which is the main contribution of the present article.

\vspace{.1in}

\noindent \textbf{Notation.} Throughout the article, we write $A\lesssim B$ if $A\leq CB$ for some absolute constant $C$; $A\sim B$ if $A\lesssim B$ and $B\lesssim A$; $A\lesssim_\epsilon B$ if $A\leq C_\epsilon B$ for all $\epsilon>0$; $A\lessapprox B$ if $A\leq C_\epsilon R^\epsilon B$ for any $\epsilon>0$, $R>1$.

For a large parameter $R$, ${\rm RapDec}(R)$ denotes those quantities that are bounded by a huge (absolute) negative power of $R$, i.e. ${\rm RapDec}(R) \leq C_N R^{-N}$ for arbitrarily large $N>0$. Such quantities are negligible in our argument. We say a function is essentially supported in a region if (the appropriate norm of) the tail outside the region is ${\rm RapDec}(R)$ for the underlying parameter $R$.

\begin{acknowledgement}
XD is supported by NSF DMS-1856475. AI was partially supported by NSF HDR TRIPODS 1934985. YO is supported by NSF DMS-1854148. HW is funded by the S.S. Chern Foundation and NSF DMS-1638352. RZ is supported by NSF DMS-1856541.  We would like to thank Pablo Shmerkin for pointing out a minor issue in a previous version regarding the pushforward measure under the orthogonal projection.
\end{acknowledgement}


\section{Setup and main estimates}
\setcounter{equation}0
In this section, we set up the problem and outline two main estimates, from which Theorem \ref{thm: pinned} follows.

Let $E\subset \mathbb{R}^d$ be a compact set with positive $\alpha$-dimensional Hausdorff measure, $\frac{d}{2}<\alpha<\frac{d}{2}+1$. Without loss of generality, assume that $E$ is contained in the unit ball. Then there exists a probability measure $\mu$ supported on $E$ such that
\begin{equation}\label{eqn: dim}
\mu(B(x,r))\lesssim r^\alpha,\qquad \forall x\in \mathbb{R}^d,\, \forall r>0.
\end{equation}

According to results on Hausdorff dimension of projections proved in \cite{KM} (also presented in \cite[Section 5.3]{M2}), there exists a ($\frac{d}{2}+1$)-dimensional subspace $V\subset\mathbb{R}^d$ such that $(\pi_V)_\ast(\mu)$, the pushforward measure of $\mu$ under the orthogonal projection from $\mathbb{R}^d$ onto $V$, is still $\alpha$-dimensional, in the sense that
\begin{equation}\label{eqn: proj energy}
I_{\beta} ((\pi_V)_*(\mu)) <\infty, \, \, \, \forall \, 0<  \beta<\alpha,
\end{equation}
where $I_{\beta} (\nu): = \int \int |x-y|^{-\beta} d\nu(x) d\nu(y)$ denotes the energy of $\nu$.
Since $V$ will be fixed throughout the proof, in the following we will drop it from the notation and write $\pi=\pi_V$ for short.

Similarly as in \cite{GIOW}, it will be helpful to consider two disjoint subsets $E_1, E_2$ of $E$ and focus on showing that the distance set between $E_1, E_2$ already has positive Lebesgue measure.

The two subsets will be chosen as follows. First, it is elementary that one can find two subsets $\tilde{E}_1, \tilde{E}_2\subset \pi(E)$ with positive projected measure satisfying $d(\tilde{E}_1,\tilde{E}_2)\gtrsim 1$. For $i=1,2$, let
\[
\tilde{\mu}_i:=\pi_\ast(\mu)(\tilde{E}_i)^{-1} \pi_\ast(\mu)\chi_{\tilde{E}_i}.
\]
It is easy to see that $\tilde{\mu}_i$ is a probability measure supported on $\tilde{E}_i$ satisfying $I_{\beta}(\tilde{\mu}_i)<\infty$, $\forall 0<\beta<\alpha$.

Next, define $E_i:=\pi^{-1}(\tilde{E}_i)\cap E$, $i=1,2$. Then one has $E_1, E_2\subset E$ and $d(E_1, E_2)\gtrsim 1$. Moreover, letting $\mu_i:=\mu(E_i)^{-1}\mu \chi_{E_i}$, $i=1,2$, one obtains a pair of probability measures $\mu_1, \mu_2$ that are supported on $E_1, E_2$ respectively, satisfying $\mu_i(B(x,r))\lesssim r^\alpha$, $\forall x\in \mathbb{R}^d, \forall r>0$, $i=1,2$. It is straightforward to check that $\pi_*(\mu_i)=\tilde\mu_i$.
Our goal in the following is to show that when $\alpha>\frac{d}{2}+\frac{1}{4}$, there exits $x\in E_2$ such that $|\Delta_x(E_1)|>0$. Note that in the above, for the orthogonal projection $\pi$ to be well defined, we have already used the assumption that $d$ is an even integer.

To relate the measures discussed above to the distance set, it is useful to consider their pushforward measures under the distance map. More precisely, let $x\in E_2$ be any fixed point and let $d^x(y):=|x-y|$ be its induced distance map. Then, the pushforward measure $d^x_\ast(\mu_1)$, defined as
\[
\int_{\mathbb{R}} \psi(t)\,d^x_\ast(\mu_1)(t)=\int_{E_1}\psi(|x-y|)\,d\mu_1(y),
\]
is a natural measure that is supported on $\Delta_x(E_1)$.

In the following, we will construct another complex valued measure $\mu_{1,g}$ that is the \emph{good} part of $\mu_1$ with respect to $\mu_2$, and study its pushforward under the map $d^x$. The main estimates are the following.

\begin{proposition} \label{mainest1} Let $d\geq 4$ be an even integer and $\alpha > \frac{d}{2}$. If we choose $R_0$ large enough in the construction of $\mu_{1,g}$ in Section \ref{sec: good} below, then there is a subset $E_2' \subset E_2$ so that $\mu_2(E_2') \ge 1 - \frac{1}{1000}$ and for each $x \in E_2'$,
$$	
	\| d^x_*(\mu_1) - d^x_*(\mu_{1,g}) \|_{L^1} < \frac{1}{1000}.
$$
\end{proposition}

\begin{proposition} \label{mainest2} Let $d\geq 4$ be an even integer and $\alpha > \frac{d}{2}+\frac{1}{4}$, then for sufficiently small $\delta$ in terms of $\alpha$ in the construction of $\mu_{1,g}$ in Section \ref{sec: good} below,
$$	
\int_{E_2}  \| d^x_*(\mu_{1,g}) \|_{L^2}^2 d \mu_2(x) < + \infty.
$$
\end{proposition}

In the above propositions, we have slightly abused notation by using $d^x_*(\mu)$ and $d^x_*(\mu_{1,g})$ to denote both the pushforward measures and their densities. To be completely
rigorous, one would need to define the density as limit of approximate identity,
then derive the propositions above uniformly with respect to the limiting
process. We omit the details as the process is fairly standard (for example see
\cite{Liu}).

It is a routine exercise to check that Theorem \ref{thm: pinned} is immediately implied by the two propositions above.

\begin{proof} [Proof of Theorem \ref{thm: pinned} using Proposition \ref{mainest1} and Proposition \ref{mainest2}] The two propositions tell us that there is a point $x \in E_2$ so that
\begin{equation} \label{L1close}
\| d^x_*(\mu_1) - d^x_*(\mu_{1,g}) \|_{L^1} < 1/1000, \end{equation}
and
\begin{equation} \label{L2bound}
\| d^x_*(\mu_{1,g}) \|_{L^2} < + \infty.
\end{equation}
Since $d^x_*(\mu_1)$ is a probability measure, (\ref{L1close}) guarantees that $\|d^x_*(\mu_{1,g})\|_{L^1} \ge 1 - 1/1000. $
Note that the support of $d^x_*(\mu_1)$ is contained in $\Delta_x(E)$.  Therefore
$$
\int_{\Delta_x(E)} | d^x_* \mu_{1,g} | = \int |d^x_* (\mu_{1,g})| - \int_{\Delta_x(E)^c} |d^x_*(\mu_{1,g})|
$$
$$
\ge 1 - \frac{1}{1000} - \int |d^x_*(\mu_1) - d^x_*(\mu_{1,g})| \ge 1 - \frac{2}{1000}.
$$
But on the other hand,
\begin{equation} \label{yes}
\int_{\Delta_x(E)} | d^x_* \mu_{1,g}| \le | \Delta_x(E)|^{1/2} \left( \int |d^x_* \mu_{1,g}|^2 \right)^{1/2}.
\end{equation}
Since (\ref{L2bound}) tells us that $\int |d^x_* \mu_{1,g}|^2$ is finite, it follows that $|\Delta_x(E)|$ is positive.  \end{proof}

Note that the two propositions in the above are parallel to \cite[Proposition 2.1, 2.2]{GIOW}. The main novelty of our proof is the construction of the good measure $\mu_{1,g}$ and the justification of Proposition \ref{mainest1}. Once that step is completed, the proof of Proposition \ref{mainest2} proceeds very similarly to its corresponding version in \cite{GIOW}.

\section{Construction of good measure and Proposition \ref{mainest1}}
\setcounter{equation}0

\subsection{Construction of good measure}\label{sec: good}

Our plan is to define $\mu_{1,g}$ by eliminating certain bad wave packets from $\mu_1$. We will show that this can be done at a single scale at each time and the error between $\mu_{1,g}$ and $\mu_1$ has sufficient decay. This procedure proceeds very similarly to \cite{GIOW}. Heuristically, we would like to define a wave packet to be bad if its projection onto the $(\frac{d}{2}+1)$-dimensional subspace $V$ has $\tilde\mu_2$-mass that is significantly higher than average.

Here are the details. Let $R_0$ be a large number that will be determined later, and let $R_j = 2^j R_0$. In $\ZR^d$, cover the annulus $R_{j-1} \le |\omega| \le R_j$ by rectangular blocks $\tau$ with dimensions approximately $R_j^{1/2} \times \cdots \times R_j^{1/2} \times R_j$, with the long direction of each block $\tau$ being the radial direction.  Choose a smooth partition of unity subordinate to this cover such that
$$
1 = \psi_0 + \sum_{j \ge 1, \tau} \psi_{j, \tau},
$$
where $\psi_0$ is supported in the ball $B(0,2R_0)$.

Let $\delta > 0$ be a small constant that we will choose later. For each $(j, \tau)$, cover the unit ball in $\ZR^d$ with tubes $T$ of dimensions approximately $R_j^{-1/2 + \delta} \times\cdots \times R_j^{-1/2+\delta} \times 2$ with the long axis parallel to the long axis of $\tau$. The covering has uniformly bounded overlap, each $T$ intersects at most $C(d)$ other tubes. We denote the collection of all these tubes as $\mathbb{T}_{j, \tau}$. Let $\eta_T$ be a smooth partition of unity subordinate to this covering, so that for each choice of $j$ and $\tau$, $ \sum_{T \in \mathbb{T}_{j, \tau}} \eta_T $ is equal to 1 on the ball of radius 2 and each $\eta_T$ is smooth.

For each $T \in \mathbb{T}_{j, \tau}$, define an operator
\[
M_T f := \eta_T (\psi_{j, \tau} \hat f)^{\vee},
\]
which, morally speaking, maps $f$ to the part of it that has Fourier support in $\tau$ and physical support in $T$.  Define also $M_0 f := (\psi_0 \hat f)^{\vee}$.  We denote $\mathbb{T}_j = \cup_{\tau} \mathbb{T}_{j, \tau}$ and $\mathbb{T} = \cup_{j \ge 1} \mathbb{T}_j$. Hence, for any $L^1$ function $f$ supported on the unit ball, one has the decomposition
\[
f = M_0 f + \sum_{T \in \mathbb{T}} M_T f+\text{RapDec}(R_0)\|f\|_{L^1}.
\]
See \cite[Lemma 3.4]{GIOW} for a justification of the above decomposition. (Even though \cite[Lemma 3.4]{GIOW} is stated in two dimensions, the argument obviously extends to higher dimensions.)

Let $c(\alpha)=c(\alpha, d)> 0$ be a large constant to be determined later, and let $4T$ denote the concentric tube of four times the radius. We say a tube $T \in \mathbb{T}_{j, \tau}$ is \emph{bad} if
\[
\mu_2 (4T) \ge R_j^{-d/4 + c(\alpha) \delta}.
\]

Note that the above definition is completely parallel to the one used in \cite{GIOW}, with the only difference being the choice of the mass threshold $R_j^{-d/4 + c(\alpha) \delta}$ for a tube $T$ to be bad. This threshold is carefully chosen so that the error estimate, i.e. proof of Proposition \ref{mainest1}, below can work through.

A tube $T$ is \emph{good} if it is not bad, and we define
\[
\mu_{1,g}:=M_0 \mu_1 + \sum_{T \in \mathbb{T}, T \textrm{ good}} M_T \mu_1.
\]
We point out that $\mu_{1,g}$ is only a complex valued measure, and is essentially supported in the $R_0^{-1/2+\delta}$-neighborhood of $E_1$ with a rapidly decaying tail away from it (see Lemma \cite[Lemma 5.2]{GIOW} for a proof, which is presented in the two dimensional case but works in all dimensions).

\subsection{Proof of Proposition \ref{mainest1}}
We would like to relate $\| d^x_*(\mu_1) - d^x_*(\mu_{1,g}) \|_{L^1}$ to the geometry of bad tubes. To start with, recall the following lemma:
\begin{lemma}\label{lem: bad}{\cite[Lemma 3.5]{GIOW}}
For any point $x\in E_2$, define
\[
\text{Bad}_j(x) := \bigcup_{T \in \mathbb{T}_j: x \in 2T \textrm{ and $T$ is bad}} 2T,\quad \forall j\geq 1.
\]
Then there holds
\[
 \| d^x_*(\mu_{1,g}) - d^x_*(\mu_1) \|_{L^1} \lesssim \sum_{j \ge 1} R_j^{\delta}\mu_1 ( \text{Bad}_j(x)) + \text{RapDec}(R_0).
\]
\end{lemma}

Note that the proof of this lemma has nothing to do with the actual definition of bad tubes and the ambient dimension of the space, so it applies directly to our setting. To estimate the measure of $\text{Bad}_j(x)$, define
$$
\text{Bad}_j := \{ (x_1, x_2)\in E_1\times E_2: \textrm{ there is a bad } T \in \mathbb{T}_j \textrm{ so that } 2T \textrm{ contains } x_1 \textrm{ and } x_2 \}.
$$
We claim that Proposition \ref{mainest1} would follow if one can show for a sufficiently large constant $c(\alpha)>0$ that
\begin{equation}\label{eqn: bad}
    \mu_1 \times \mu_2(\text{Bad}_j) \lesssim R_j^{- 2 \delta},\quad \forall j\geq 1.
\end{equation}

Indeed, since
\[
\mu_1 \times \mu_2 (\text{Bad}_j) = \int \mu_1(\text{Bad}_j(x)) d\mu_2(x),
\]
the estimate (\ref{eqn: bad}) ensures that there exists $B_j \subset E_2$ so that $ \mu_2(B_j) \le R_j^{- (1/2) \delta}$ and for all $x \in E_2 \setminus B_j$,
\[
\mu_1(\text{Bad}_j(x)) \lesssim R_j^{- (3/2) \delta}.
\]
Let $E_2' = E_2 \setminus \bigcup_{j \ge 1} B_j$ and choose $R_0$ sufficiently large (depending on $\delta$ and $\alpha$). One obviously has $\mu_2(E_2') \ge 1 - \frac{1}{1000}$, and for each $x \in E_2'$ the bound
\[
\| d^x_*(\mu_{1,g}) - d^x_*(\mu_1) \|_{L^1} \lesssim R_0^{- (1/2) \delta} \le \frac{1}{1000}\,,
\]
according to Lemma \ref{lem: bad}.

In order to prove estimate (\ref{eqn: bad}), we apply the following radial projection theorem of Orponen \cite{O17b}. The choice of the threshold in the definition of bad tubes in the above will play an important role in this step. In order to state the theorem, we first define a radial projection map $P_y: \mathbb{R}^n \setminus \{ y \} \rightarrow S^{n-1}$ by
$$
P_y(x) = \frac{x-y}{|x-y|}.
$$

\begin{theorem}\label{thm: Orponen}{\cite[Orponen]{O17b}}
For every $\beta > n-1$ there exists $p(\beta)> 1$ so that the following holds.  Suppose that $\nu_1$ and $\nu_2$ are measures on the unit ball in $\mathbb{R}^n$ with disjoint supports and that $I_{\beta}(\nu_i )<\infty $.  Then
	$$
	\int \| P_y \nu_2 \|_{L^p}^p d \nu_1(y) < \infty.
	$$

\end{theorem}

Note that we cannot apply the above theorem directly to our problem in $\mathbb{R}^d$, because the measures $\mu_1, \mu_2$ we are dealing with have dimension $\alpha$ that is barely larger than $\frac{d}{2}$ (hence fails to satisfy $\alpha>d-1$). This motivates us to consider the projected measures $\tilde\mu_i=\pi_\ast(\mu_i)$ instead.

Recall from the definition that for $i=1,2$, $\tilde{\mu}_i$ is a measure on the $(\frac{d}{2}+1)$-dimensional subspace $V\subset \mathbb{R}^d$ and satisfies $I_{\beta} (\tilde\mu_i) <\infty$ for any $0<\beta<\alpha$. Whenever $\alpha>\frac{d}{2}$, one has $\alpha>(\frac{d}{2}+1)-1$. Therefore, Theorem \ref{thm: Orponen} does apply to $\tilde\mu_1, \tilde\mu_2$, and one has
\begin{equation}\label{eqn: radial}
\int \| P_y \tilde\mu_2 \|_{L^p}^p d \tilde\mu_1(y) <  \infty.
\end{equation}

To prove estimate (\ref{eqn: bad}), we first define a set  $\widetilde{\text{Bad}}_j$ in $V^2$.

We have chosen the sets $E_1, E_2$ at the beginning such that $d(E_1, E_2)\gtrsim 1$ and $d(\pi(E_1),\pi(E_2))\gtrsim 1$. By definition of $\text{Bad}_j$, it suffices to consider tubes $T\in \mathbb{T}_j$ that intersect both $E_1$ and $E_2$. Hence, the projected tube $\pi(T)\subset V$ also looks like a tube, with side length $\sim 1$ in the long direction, and $\sim R_j^{-\frac{1}{2}+\delta}$ in the rest of the directions. Therefore, $\mathbb{T}_j$ gives rise to a collection $\tilde{\mathbb{T}}_j$ that contains tubes in $V$ of dimensions roughly $1\times R_j^{-\frac{1}{2}+\delta}\times\cdots \times R_j^{-\frac{1}{2}+\delta}$.

One can similarly define a tube $\tilde{T}\in \tilde{\mathbb{T}}_j$ to be bad if $\tilde\mu_2(4\tilde{T})\geq R_j^{-\frac{d}{4}+c(\alpha)\delta}$. It is easy to see that the badness of a tube is preserved under the projection. Indeed, if $T\in \mathbb{T}_j$ is bad, then
\[
\tilde\mu_2(4\pi(T))\geq \mu_2(4T)\geq R_j^{-\frac{d}{4}+c(\alpha)\delta}.
\]

Define
\[
\widetilde{\text{Bad}}_j:=\{(x_1, x_2)\in V^2: \textrm{ there is a bad } \tilde{T} \in \tilde{\mathbb{T}}_j \textrm{ so that } 2\tilde{T} \textrm{ contains } x_1 \textrm{ and } x_2\}.
\]
Then one has
\[
\mu_1 \times \mu_2( \text{Bad}_j)\leq \tilde\mu_1\times \tilde\mu_2(\widetilde{\text{Bad}}_j)=\int \tilde\mu_2(\widetilde{\text{Bad}}_j(y)) d \tilde\mu_1(y),
\]
where
\[
\widetilde{\text{Bad}}_j(y) := \bigcup_{\tilde{T} \in \tilde{\mathbb{T}}_j: y \in 2\tilde{T} \textrm{ and $\tilde{T}$ is bad}} 2\tilde{T}.
\]

With bound (\ref{eqn: radial}), the desired estimate (\ref{eqn: bad}) follows by an argument identical to \cite[Proof of Lemma 3.6]{GIOW}. We sketch the argument here for the sake of completeness.

Let $\tilde{T} \in \tilde{\mathbb{T}}_j$ be a bad tube and $y \in 2\tilde{T}\cap \pi(E_1)$.  Let $A(\tilde{T})$ be the cap of the sphere $S^{\frac{d}{2}}$ whose center corresponds to the direction of the long axis of $\tilde{T}$ and with radius $\sim  R_j^{-1/2 + \delta}$.  Since $d(\pi(E_1), \pi(E_2))\gtrsim 1$, one has $P_y (4\tilde{T}\cap \pi(E_2)) \subset A(\tilde{T})$, hence

\begin{equation} \label{bigarc} P_y \tilde{\mu}_2 (A(\tilde{T})) \ge \tilde\mu_2(4\tilde{T}) \ge R_j^{-\frac{d}{4} + c(\alpha) \delta}.  \end{equation}

 Therefore, $P_y( \widetilde{\text{Bad}}_{j}(y))$ can be covered by caps $A(\tilde{T})$ of radius $\sim R_j^{-1/2 + \delta}$ which each satisfies (\ref{bigarc}).  By the Vitali covering lemma, there exists a disjoint subset of $A(\tilde{T})$ so that $5 A(\tilde{T})$ covers $P_y( \widetilde{\text{Bad}}_{j}(y))$.  Hence, the total number of disjoint $A(\tilde{T})$ in the covering is bounded by $R_j^{\frac{d}{4}-c(\alpha)\delta}$, which implies
\[
| P_y(\widetilde{\text{Bad}}_{j}(y)) |\lesssim R_j^{\frac{d}{4}-c(\alpha)\delta}\cdot R_j^{\frac{d}{2}(-1/2+\delta)}  = R_j^{-(c(\alpha)-\frac{d}{2}) \delta}, \]where $|\cdot|$ denotes the surface measure on $S^{\frac{d}{2}}$. Therefore, by H\"older's inequality and by choosing $c(\alpha)$ sufficiently large, one has
\[
\begin{split}
\mu_1 \times \mu_2(\text{Bad}_j) \leq & \int \tilde\mu_2 (\widetilde{\text{Bad}}_{j}(y)) d\tilde\mu_1(y)  \le \int \left( \int_{P_y (\widetilde{\text{Bad}}_{j}(y)) } P_y\tilde\mu_2 \right) d \tilde\mu_1(y)\\
\leq & \sup_y | P_y(\widetilde{\text{Bad}}_{j}(y)) |^{1 - \frac{1}{p}} \int \| P_y \tilde\mu_2 \|_{L^p} d \tilde\mu_1 \lesssim R_j^{- 2 \delta}.
\end{split}
\]
This completes the justification of (\ref{eqn: bad}) thus the proof of Proposition \ref{mainest1}.

\section{Refined decoupling and Proposition
\setcounter{equation}0
\ref{mainest2}}

In this section, we prove Proposition \ref{mainest2}, which will complete the proof of Theorem \ref{thm: pinned}. This part of the argument proceeds very similarly as \cite[Proof of Proposition 2.2]{GIOW}, with the only difference being the change of the definition of good tubes.

Let $\si_r$ be the normalized surface measure on the sphere of radius $r$. The main estimate in the proof of Proposition \ref{mainest2} is the following:

\begin{lemma}\label{lem: decbound}
For any $\alpha > 0$, $r>0$, and $\delta$ sufficiently small depending on $\alpha, \epsilon$:
	$$
	\int_{E_2} | \mu_{1,g} * \hat \sigma_r(x)|^2 d \mu_2(x) \le C(R_0) r^{-\frac{d}{2(d+1)}-\frac{(d-1)\alpha}{d+1}+\epsilon} r^{-(d-1)} \int  | \hat \mu_1 |^2 \psi_r d \xi+{\rm RapDec}(r),
	$$
where $\psi_r$ is a weight function which is $\sim 1$ on the annulus $r-1 \le |\xi| \le r+1$ and decays off of it.  To be precise, we could take
$$
\psi_r(\xi) = \left( 1 + | r- |\xi|| \right)^{-100}.
$$

\end{lemma}

To see how this lemma implies the desired estimate in Proposition \ref{mainest2}, one first observes that
$$
d^x_*(\mu_{1,g})(t)=t^{d-1}\mu_{1,g}*\si_t(x)\,.
$$
Since $\mu_{1,g}$ is essentially supported in the $R_0^{-1/2+\delta}$-neighborhood of $E_1$, for $x\in E_2$, we only need to consider $t\sim 1$. Hence,
\[
\begin{split}
\int_{E_2} \|d_*^x(\mu_{1,g})\|_{L^2}^2\,d\mu_2(x)\lesssim &\int_0^\infty \int_{E_2}|\mu_{1,g}*\sigma_t(x)|^2\,d\mu_2(x)t^{d-1}\,dt\\
\sim & \int_0^\infty \int_{E_2}|\mu_{1,g}*\hat{\sigma}_r(x)|^2\,d\mu_2(x)r^{d-1}\,dr,
\end{split}
\]
where in the second step, we have used a limiting process and an $L^2$-identity proved by Liu \cite[Theorem 1.9]{Liu}: for any Schwartz function $f$ on $\mathbb R^d, d\geq 2$, and any $x\in \mathbb R^d$,
$$
\int_0^\infty |f*\sigma_t(x)|^2\, t^{d-1}\,dt = 
\int_0^\infty |f*\hat\sigma_r(x)|^2\, r^{d-1}\,dr.$$

Applying Lemma \ref{lem: decbound} for each $r>0$ and dropping the rapidly decaying tail as we may, one can bound the above further by
\[
\begin{split}
&\lesssim_{R_0} \int_0^\infty \int_{\mathbb{R}^d}  r^{-\frac{d}{2(d+1)}-\frac{(d-1)\alpha}{d+1}+\epsilon} \psi_r(\xi) | \hat \mu_1(\xi)|^2 \,d \xi dr  \\
&\lesssim  \int_{\mathbb{R}^d} |\xi|^{-\frac{d}{2(d+1)}-\frac{(d-1)\alpha}{d+1}+\epsilon} | \hat \mu_1 (\xi)|^2 \,d \xi \sim I_{\beta} (\mu_1),
\end{split}
\]where $\beta=d-\frac{d}{2(d+1)}-\frac{(d-1)\alpha}{d+1}+\epsilon$, by a Fourier representation for $I_\beta$ (cf. Proposition 8.5 of \cite{W}):
\[
I_\beta (\mu) = \int |x-y|^{-\beta} d\mu(x) d\mu(y)=c_{d, \beta}  \int_{\mathbb{R}^d} |\xi|^{-(d - \beta)} |\hat \mu(\xi) |^2 \,d \xi.
\]
One thus has $I_\beta(\mu_1)<\infty$ if $\beta<\alpha$, which is equivalent to $\alpha>\frac{d}{2}+\frac{1}{4}$. The proof of Proposition \ref{mainest2} is thus complete upon verification of Lemma \ref{lem: decbound}.

\subsection{Refined decoupling estimates} The key ingredient in the proof of Lemma \ref{lem: decbound} is the following refined decoupling theorem, which is derived by applying the $l^2$ decoupling theorem of Bourgain and Demeter \cite{BD} at many different scales.

Here is the setup. Suppose that $S \subset \mathbb{R}^d$ is a compact and strictly convex $C^2$ hypersurface with Gaussian curvature $\sim 1$. For any $\epsilon>0$, suppose there exists $0<\delta\ll \epsilon$ satisfying the following.   Suppose that the 1-neighborhood of $R S$ is partitioned into $R^{1/2} \times ... \times R^{1/2} \times 1$ blocks $\theta$.  For each $\theta$, let $\mathbb{T}_\theta$ be a set of tubes of dimensions $R^{-1/2 + \delta} \times 1$ with long axis perpendicular to $\theta$, and let $\mathbb{T} = \cup_\theta \mathbb{T}_\theta$. Each $T\in \ZT$ belongs to $\ZT_{\theta}$ for a single $\theta$, and we let $\theta(T)$ denote
this $\theta$. We say that $f$ is microlocalized to $(T,\theta(T))$ if $f$ is essentially supported in
$2T$ and $\hat{f}$ is essentially supported in $2\theta(T)$.

\begin{theorem}\label{thm: dec}{\cite[Corollary 4.3]{GIOW}}
Let $p$ be in the range $2 \le p \le \frac{2(d+1)}{d-1}$.  For any $\epsilon>0$, suppose there exists $0<\delta\ll\epsilon$ satisfying the following. Let $\mathbb{W} \subset \mathbb{T}$ and suppose that each $T \in \mathbb{W}$ lies in the unit ball.  Let $W = | \mathbb{W}|$.  Suppose that $f = \sum_{T \in \mathbb{W}} f_T$, where $f_T$ is microlocalized to $(T, \theta(T))$.  Suppose that $\| f_T \|_{L^p}$ is $\sim$ constant for each $T \in \mathbb{W}$.   Let $Y$ be a union of $R^{-1/2}$-cubes in the unit ball each of which intersects at most $M$ tubes $T \in \mathbb{W}$.  Then
	$$ \| f \|_{L^p(Y)} \lesssim_\epsilon R^\epsilon \left(\frac{M}{W} \right)^{\frac{1}{2} - \frac{1}{p}} \left(\sum_{T \in \mathbb{W}} \| f_T \|_{L^p}^2 \right)^{1/2}. $$
\end{theorem}

\subsection{Proof of Lemma \ref{lem: decbound}}

Assume $r>10R_0$ (we omit the $r<10R_0$ case, which is much easier and can be dealt with by the same argument at the end of Section 5 in \cite{GIOW}). By definition, 
\[
\mu_{1,g}* \hat\sigma_r=\sum_{R_j \sim r} \sum_\tau \sum_{T\in \mathbb{T}_{j, \tau}: T \textrm{ good}} M_T \mu_1 * \hat \sigma_r+{\rm RapDec}(r).
\]

The contribution of ${\rm RapDec}(r)$ is already taken into account in the statement of Lemma 4.1. Hence without loss of generality we may ignore the tail ${\rm RapDec}(r)$ in the argument below.

Let $\eta_1$ be a bump function adapted to the unit ball and define
\[
f_T = \eta_1 \left( M_T \mu_1 * \hat \sigma_r \right).
\]One can easily verify that $f_T$ is microlocalized to $(T,\theta(T))$.

Let $p=\frac{2(d+1)}{d-1}$. After dyadic pigeonholing, there exists $\lambda>0$ such that
\[
\int |\mu_{1,g}*\hat\sigma_r(x)|^2\,d\mu_2(x)\lesssim \log r \int | f_\lambda (x) |^2 d\mu_2(x),
\]where
\[
f_\lambda=\sum_{T\in \mathbb{W}_\lambda}f_T,\quad \mathbb{W}_\lambda:= \bigcup_{R_j\sim r}\bigcup_{\tau}\Big\{ T\in \ZT_{j,\tau}: T \text{ good }, \| f_T \|_{L^p} \sim \lambda \Big\}.
\]
To simplify the argument, we do another pigeonholing: divide the unit ball into $r^{-1/2}$-cubes $q$ and sort them. This then reduces the integration domain of $|f_\lambda|^2$ in the above to $Y_{ M} = \bigcup_{q \in \mathcal{Q}_{ M}} q$ for some $M$, where
\[
\mathcal{Q}_M:=\{ r^{-1/2}\textrm{-cubes } q: q \textrm{ intersects } \sim M \textrm{ tubes } T \in \mathbb{W}_\lambda \}.
\]

Since $f_\lambda$ only involves good wave packets, by considering the quantity
$$
\sum_{q\in \mathcal{Q}_M} \sum_{T\in \mathbb{W}_\la: T\cap q\neq \emptyset} \mu_2(q),
$$
we get
\begin{equation}\label{eqn: count}
    M \mu_2 (\mathcal{N}_{r^{-1/2}}(Y_M) ) \lesssim  | \mathbb{W}_\lambda | r^{-\frac{d}{4} + c(\alpha) \delta},
\end{equation}
where $\mathcal{N}_{r^{-1/2}}(Y_M)$ is the $r^{-1/2}$-neighborhood of $Y_M$.

The rest of the proof of Lemma \ref{lem: decbound} will follow from Theorem \ref{thm: dec} and estimate (\ref{eqn: count}).

By H\"older's inequality and the observation that $f_\la$ has Fourier support in the $1$-neighborhood of the sphere of radius $r$, one has
\[
\int_{Y_M}|f_\lambda(x)|^2\,d\mu_2(x)\lesssim \left(\int_{Y_{M}} | f_\lambda|^p \right)^{2/p} \left(\int_{Y_{M}} |\mu_2 * \eta_{1/r}|^{p/(p-2)} \right)^{1-2/p},
\]where $\eta_{1/r}$ is a bump function with integral $1$ that is essentially supported on the ball of radius $1/r$.

To bound the second factor, we note that $\eta_{1/r} \sim r^d$ on the ball of radius $1/r$ and rapidly decaying off it. Using the fact that $\mu_2(B(x,r))\lesssim r^\alpha, \forall x\in \ZR^d, \forall r>0$, we have
$$
\|\mu_2*\eta_{1/r}\|_\infty \lesssim r^{d-\al}\,.
$$
Therefore,
\[
\begin{split}
\int_{Y_{M}} |\mu_2 * \eta_{1/r}|^{p/(p-2)} \lesssim & \|\mu_2 * \eta_{1/r}\|_{\infty}^{2/(p-2)} \int_{Y_M} d\mu_2*\eta_{1/r}\\
\lesssim  &r^{2(d-\alpha)/(p-2)}\mu_2(\mathcal{N}_{r^{-1/2}}(Y_M)).
\end{split}
\]By Theorem \ref{thm: dec}, the first factor can be bounded as follows:
\[
\begin{split}
\left(\int_{Y_{M}} | f_\lambda|^p \right)^{2/p}\lessapprox & \left(\frac{M}{\mathbb{W}_\lambda}\right) ^{1-2/p}\sum_{T\in\mathbb{W}_\lambda}\|f_T\|_{L^p}^2\\
\lesssim & \left(\frac{r^{-\frac{d}{4} + c(\alpha) \delta}}{\mu_2(\mathcal{N}_{r^{-1/2}}(Y_{M}))} \right)^{1-2/p} \sum_{T \in \mathbb{W}_\lambda} \| f_T \|_{L^p}^2,
\end{split}
\]where the second step follows from (\ref{eqn: count}).

Combining the two estimates together, one obtains
\[
\int_{Y_M}|f_\lambda(x)|^2\,d\mu_2(x) \lesssim r^{O_\alpha(\delta)+(\frac{5}{2p}-\frac{1}{4})d-\frac{2\alpha}{p}}\sum_{T \in \mathbb{W}_\lambda} \| f_T \|_{L^p}^2.
\]
Observe that $\|f_T\|_{L^p}$ has the following simple bound:
\[
\begin{split}
\|f_T\|_{L^p}\lesssim & \|f_T\|_{L^\infty}|T|^{1/p}\lesssim \sigma_r(\theta(T))^{1/2}|T|^{1/p} \|\widehat{M_T\mu_1}\|_{L^2(d\sigma_r)}\\
= & r^{-(\frac{1}{2p}+\frac{1}{4})(d-1)+O_\alpha(\delta)}\|\widehat{M_T\mu_1}\|_{L^2(d\sigma_r)}.
\end{split}
\]Plugging this back into the above formula, one obtains
\[
\begin{split}
\int_{Y_M}|f_\lambda(x)|^2\,d\mu_2(x) \lesssim & r^{O_\alpha(\delta)+(\frac{3}{2p}-\frac{3}{4})d-\frac{2\alpha}{p}+\frac{1}{p}+\frac{1}{2}}\sum_{T \in \mathbb{W}_\lambda}\|\widehat{M_T\mu_1}\|_{L^2(d\sigma_r)}^2\\
\lesssim & r^{-\frac{d}{2(d+1)}-\frac{(d-1)\alpha}{d+1}+\epsilon} r^{-(d-1)} \int | \hat \mu_1|^2 \psi_r \,d \xi,
\end{split}
\]where $p=2(d+1)/(d-1)$ and we have used orthogonality and chosen $\delta$ sufficiently small depending on $\alpha$, $\epsilon$. The proof of Lemma \ref{lem: decbound} and hence Proposition \ref{mainest2} is complete.

\section{Further comments}

\subsection{Generalization to other norms}\label{sec: norm}

Similarly as the two-dimensional case in \cite{GIOW}, Theorem \ref{main} and \ref{thm: pinned} still hold if $\Delta(E)$ and $\Delta_x(E)$ are replaced by
$$\Delta^K(E)=\left\{{||x-y||}_{K}: x,y \in E \right\}$$ and
$$\Delta^K_x(E)=\left\{{||x-y||}_K: y \in E \right\}$$ respectively, where $K$ is a symmetric convex body whose boundary $\partial K$ is $C^\infty$ smooth and has everywhere positive Gaussian curvature, and ${|| \cdot ||}_K$ is the distance induced by the norm determined by $K$.

The argument is identical to the one given in Section 7 of \cite{GIOW}, where the main additional ingredient is the celebrated stationary phase formula due to Herz \cite{Herz}, which says
$$ \widehat{\sigma}_K(\xi)=C\left(\frac{\xi}{|\xi|}\right){|\xi|}^{-\frac{d-1}{2}}\left(\cos \left(2\pi \left(\|\xi\|_{K^*}-\frac{d-1}{8}\right)\right)\right),$$ where $\sigma_K$ is the normalized surface measure on
$$S=\{x \in \mathbb{R}^d: \|x\|_K=1\}$$ and $\|\cdot\|_{K^{*}}$ is the dual norm defined by
$$ \|\xi\|_{K^*}=\sup_{x \in K} x \cdot \xi.$$ We omit the details.





\subsection{Why our method fails in odd dimensions}
In odd dimension $d$, in order to make use of the Orponen's radial projection theorem to control the bad part, we project $\al$-dimensional measure $\mu$ onto a $\frac{d+1}{2}$-dimensional plane $V$, since the condition $\al>\frac{d}{2}$ only guarantees that $\al>\frac{d+1}{2}-1$. To make the proof for bad part work through, we need to choose the mass threshold for bad tubes as follows: $T\in \ZT_{j,\tau}$ is \emph{bad} if
$$
\mu_2(4T)\gtrsim R_j^{-(d-1)/4+c(\al)\delta}\,.
$$
Then the numerology for good part gives us the following dimensional threshold for Falconer's distance set problem:
$$
\frac d 2+\frac 14+\frac{1}{4d}\,,
$$
which is not as good as the previously best known result $\frac d2+\frac 14+\frac{1}{8d-4}$ from \cite{DZ}.

In fact, when $d$ is odd, there exists counterexample that prevents one from removing a larger bad part from the measure. More precisely, consider a set $E\subset \mathbb{R}^d$ that is contained in some $\frac{d+1}{2}$ dimensional subspace of $\mathbb{R}^d$ with positive $\frac{d+1}{2}$ dimensional Lebesgue measure. For instance, let $E$ be the unit ball $B^{\frac{d+1}{2}}$. Then for every $T\in \ZT_{j,\tau}$,
\[
\mu_2(T)\sim R_j^{-\frac{\frac{d+1}{2}-1}{2}+\delta}=R_j^{-\frac{d-1}{4}+\delta}.
\]Hence it is impossible to further lower the bad threshold.

\vskip.125in

\section{Connections with the Erd\H os distance problem}

\vskip.125in

The following definition is due to the second listed author, Rudnev and Uriarte-Tuero (\cite{IRU14}).

\begin{definition} \label{sadaptable} Let $P$ be a set of $N$ points contained in ${[0,1]}^d$. Define the measure
\begin{equation} \label{pizdatayamera} d \mu^s_P(x)=N^{-1} \cdot N^{\frac{d}{s}} \cdot \sum_{p \in P} \chi_B(N^{\frac{1}{s}}(x-p))\, dx, \end{equation} where $\chi_B$ is the indicator function of the ball of radius $1$ centered at the origin. We say that $P$ is \emph{$s$-adaptable} if there exists $C$ independent of $N$ such that
\begin{equation} \label{sadaptenergy} I_s(\mu_P)=\int \int {|x-y|}^{-s} \,d\mu^s_P(x) \,d\mu^s_P(y) \leq C. \end{equation}
\end{definition}

It is not difficult to check that if the points in set $P$ are separated by distance $cN^{-1/s}$, then (\ref{sadaptenergy}) is equivalent to the condition
\begin{equation} \label{discreteenergy} \frac{1}{N^2} \sum_{p \not=p'} {|p-p'|}^{-s} \leq C, \end{equation}where the exact value of $C$ may be different from line to line. In dimension $d$, it is also easy to check that if the distance between any two points of $P$ is $\gtrsim N^{-1/d}$, then (\ref{discreteenergy}) holds for any $s \in [0,d)$, and hence $P$ is $s$-adaptable.

Let $K$ be a symmetric convex body as in Section \ref{sec: norm}. We will prove that if $d$ is even and $P$ is $s$-adaptable, for all $s\in (\frac{d}{2}+\frac{1}{4}, d)$, then for some $x \in P$,
\[| \Delta^K_{x}(P)| \gtrapprox N^{\frac{1}{\frac{d}{2}+\frac{1}{4}}}.
\]

Moreover, the proof below shows that we get this many distinct $N^{-\frac{1}{s_0}}$-separated distances, where $s_0=\frac{d}{2}+\frac{1}{4}$. The best currently known bounds for distance sets in higher dimensions with respect to the Euclidean metric are due to Solymosi and Vu \cite{SV08}. While their result applies to general point sets, their exponent is smaller than ours, and their method does not yield separated distances or apply to general metrics. For the best previously known bounds in higher dimensions for general metrics, see, for example, \cite{HI05} and \cite{IL05}.

Fix $s\in (\frac{d}{2}+\frac{1}{4},d)$ and define $d\mu^s_P$ as above. Note that the support of $d\mu^s_P$ is
$\mathcal{N}_{N^{-\frac{1}{s}}}(P)$, the $N^{-\frac{1}{s}}$-neighborhood of $P$. Since $I_s(\mu_P^s)$ is uniformly bounded, the proof of (the general norm case of) Theorem \ref{thm: pinned} implies that there exists $x_0 \in \mathcal{N}_{N^{-\frac{1}{s}}}(P)$ so that
$$ {\mathcal L}(\Delta^K_{ x_0}(\mathcal{N}_{N^{-\frac{1}{s}}}(P))) \ge c>0,$$

\noindent where the constant $c$ only depends on the value of $C$ in (\ref{discreteenergy}).

Let $x$ be a point of $P$ with $|x-x_0| \le N^{-1/s}$.  It follows that for any $y$, $\| x_0 - y \|_K = \| x-y\|_K + O(N^{-1/s})$.  Let $E_{N^{-1/s}} \left(\Delta^K_{x}(P) \right)$ be the smallest number of $N^{-1/s}$-intervals needed to cover $\Delta^K_{x}(P)$.  We know that $\Delta^K_{x_0} (\mathcal{N}_{N^{-\frac{1}{s}}}(P))$ is contained in the $O(N^{-1/s})$ neighborhood of $\Delta^K_{x}(P)$, and so

$$ {\mathcal L}(\Delta^K_{ x_0}(\mathcal{N}_{N^{-\frac{1}{s}}}(P))) \lesssim N^{-\frac{1}{s}} E_{N^{-1/s}} \left(\Delta^K_{x}(P) \right). $$

Then our lower bound on ${\mathcal L}(\Delta^K_{ x_0}(\mathcal{N}_{N^{-\frac{1}{s}}}(P)))$ gives

$$E_{N^{-1/s}} \left(\Delta^K_{x}(P) \right) \gtrsim N^{1/s}. $$

In other words, $\Delta^K_{x}(P)$ contains $\gtrsim N^{1/s}$ different distances that are pairwise separated by $\gtrsim N^{-1/s}$.  In particular, $|\Delta^K_{x}(P)| \gtrsim N^{1/s}$.  Since this holds for every $s > \frac{d}{2}+\frac{1}{4}$, we get $| \Delta^K_{x}(P)| \gtrapprox N^{\frac{1}{\frac{d}{2}+\frac{1}{4}}}$ as desired.

\end{document}